\documentclass[a4paper,12pt,reqno]{amsart}

\usepackage[text={130mm,200mm}]{geometry}

\theoremstyle{plain}
\newtheorem{theorem}{Theorem}
\newtheorem*{theorem*}{Theorem}
\newtheorem{corollary}{Corollary}
\newtheorem*{corollary*}{Corollary}
\newtheorem{lemma}{Lemma}
\newtheorem*{lemma*}{Lemma}
\newtheorem{proposition}{Proposition}
\newtheorem*{proposition*}{Proposition}

\newtheorem*{conjecture*}{Conjecture}
\theoremstyle{definition}
\newtheorem{definition}{Definition}
\newtheorem*{definition*}{Definition}
\theoremstyle{remark}
\newtheorem{remark}{Remark}
\newtheorem*{remark*}{Remark}

\begin{document}

\title[Cantor series and rational numbers]{Cantor series and rational numbers}
\author{Symon Serbenyuk}
\address{Institute of Mathematics \\
 National Academy of Sciences of Ukraine \\
  3~Tereschenkivska St. \\
  Kyiv \\
  01004 \\
  Ukraine}
\email{simon6@ukr.net}



\begin{abstract}
  The article is devoted to the investigation of representation of rational numbers by Cantor series. Criteria of rationality are formulated for the case of an arbitrary sequence $(q_k)$  and some its corollaries are considered. 
 
\end{abstract}

\maketitle



\section{Introduction}

The following series 
\begin{equation}
\label{eq: Cantor series}
\frac{\varepsilon_1}{q_1}+\frac{\varepsilon_2}{q_1q_2}+...+\frac{\varepsilon_k}{q_1q_2...q_k}+...,
\end{equation}
where $Q\equiv (q_k)$ is a fixed sequence of positive integers $q_k>1$ and $\Theta_k$ is a sequence of the sets $\Theta_k\equiv\{0,1,...,q_k-1\}$ and $\varepsilon_k\in\Theta_k$, were considered by Georg Cantor  in \cite{Cantor1}. 
In the last-mentioned article necessary and sufficient conditions of expansions of rational numbers by series \eqref{eq: Cantor series} are formulated for the case of periodic sequence $(q_k)$. The problem of expansions of rational or irrational numbers by the Cantor series  have been studied by a number of researchers. For example,  P. A. Diananda, A. Oppenheim, P.~Erd\"os , J. Han\v{c}l, E. G. Straus, P. Rucki, R. Tijdeman, P. Kuhapatanakul, V. Laohakosol, Mance B., D. Marques, Pingzhi Yuan studied this problem \cite{{DO1955}}-\cite{O1954}, \cite{{S2012}, TY2002}.

The main problem of the present article is the formulation of necessary and sufficient conditions of representation of rational numbers by the Cantor series for an arbitrary sequence $(q_k)$. 
This  problem is solved using notion of  shift operator.

\section{The main definitions}
\begin{definition}
A number $x\in [0;1]$ represented by series \eqref{eq: Cantor series} is denoted by $\Delta^Q _{\varepsilon_1\varepsilon_2...\varepsilon_k...}$ and this notation is called \emph{the representation of $x$ by Cantor series \eqref{eq: Cantor series}.}
\end{definition}

\begin{definition}
A map $\sigma$ defined by the following way
$$
\sigma(x)=\sigma\left(\Delta^Q _{\varepsilon_1\varepsilon_2...\varepsilon_k...}\right)=\sum^{\infty} _{k=2}{\frac{\varepsilon_k}{q_2q_3...q_k}}=q_1\Delta^{Q} _{0\varepsilon_2...\varepsilon_k...}
$$
is called \emph{the shift operator}.
\end{definition}

It is easy to see that 
\begin{equation}
\label{eq: Cantor series 2}
\sigma^n(x)=\sigma^n\left(\Delta^Q _{\varepsilon_1\varepsilon_2...\varepsilon_k...}\right)=\sum^{\infty} _{k=n+1}{\frac{\varepsilon_k}{q_{n+1}q_{n+2}...q_k}}=q_1...q_n\Delta^{Q} _{\underbrace{0...0}_{n}\varepsilon_{n+1}\varepsilon_{n+2}...}.
\end{equation}

Therefore, 
\begin{equation}
\label{eq: Cantor series 3}
x=\sum^{n} _{i=1}{\frac{\varepsilon_i}{q_1q_2...q_i}}+\frac{1}{q_1q_2...q_n}\sigma^n(x).
\end{equation}

In \cite{S2013} the notion of shift operator of alternating Cantor series is studied in detail.

\section{Rational numbers, that have two different representations}

Certain numbers from $[0;1]$ have two different representations by Cantor series \eqref{eq: Cantor series}, i.e., 
$$
\Delta^Q _{\varepsilon_1\varepsilon_2...\varepsilon_{m-1}\varepsilon_m000...}=\Delta^Q _{\varepsilon_1\varepsilon_2...\varepsilon_{m-1}[\varepsilon_m-1][q_{m+1}-1][q_{m+2}-1]...}\equiv\sum^{m} _{i=1}{\frac{\varepsilon_i}{q_1q_2...q_i}}.
$$
\begin{theorem}
A rational number $\frac{p}{r}\in(0;1)$ has two different representations iff there exists a number $n_0$ such that $q_1q_2...q_{n_0} \equiv 0\pmod{r}$.
\end{theorem}
\begin{proof} \emph{Necessity.} Suppose
$$
\frac{p}{r}=\frac{\varepsilon_1}{q_1}+...+\frac{\varepsilon_{n-1}}{q_1q_2...q_{n-1}}+\frac{\varepsilon_n}{q_1q_2...q_n}=\frac{\varepsilon_1q_2...q_n+\varepsilon_2q_3...q_n+...+\varepsilon_{n-1}q_n+\varepsilon_n}{q_1q_2...q_n}.
$$
Whence,
$$
\varepsilon_n=\frac{pq_1q_2...q_n-r(\varepsilon_1q_2...q_n+\varepsilon_2q_3...q_n+...+\varepsilon_{n-1}q_n)}{r}.
$$

Since the conditions $\varepsilon_n\in\mathbb N$ and $(p,r)=1$ hold, we obtain $q_1q_2...q_{n} \equiv$ $\equiv 0~\pmod{r}$.

\emph{Sufficiency.} Assume that $(p,r)=1$, $p<r$ and there exists a number $n_0$ such that $q_1q_2...q_{n_0} \equiv 0\pmod{r}$; then 
$$
\frac{p}{r}=\frac{\varepsilon_1}{q_1}+\frac{\varepsilon_2}{q_1q_2}+...+\frac{\varepsilon_{n_0}}{q_1q_2...q_{n_0}}+\frac{\varepsilon_{n_0+1}}{q_1q_2...q_{n_0}q_{n_0+1}}+...=\sum^{n_0} _{i=1}{\frac{\varepsilon_i}{q_1q_2...q_i}}+t_{n_0},
$$
where $t_{n_0}$ is the residual of series, and
$$
\frac{p}{r}=\frac{\varepsilon_1q_2...q_{n_0}+\varepsilon_2q_3...q_{n_0}+...+\varepsilon_{n_0}}{q_1q_2...q_{n_0}}+t_{n_0}.
$$ 

Clearly, 
$$
p\theta=(\varepsilon_1q_2...q_{n_0}+\varepsilon_2q_3...q_{n_0}+...+\varepsilon_{n_0})+q_1q_2...q_{n_0}t_{n_0},~\mbox{where}~\theta=\frac{q_1q_2...q_{n_0}}{r}.
$$

Since $p\theta$ is a positive integer number, we have  $q_1q_2...q_{n_0}t_{n_0}=0$ or $q_1q_2...q_{n_0}t_{n_0}=~1$. That is $x=\Delta^Q _{\varepsilon_1\varepsilon_2...\varepsilon_{n_0-1}\varepsilon_{n_0}000...}$ in the first case and $x=\Delta^Q _{\varepsilon_1\varepsilon_2...\varepsilon_{n_0-1}[\varepsilon_{n_0}-1][q_{n_0+1}-1][q_{n_0+2}-1]...}$ in the second case.
\end{proof}

\section{ Representations of rational numbers}

\begin{theorem}
\label{th: the main theorem}
A number $x$ represented by series \eqref{eq: Cantor series} is  rational iff there exist numbers $n\in\mathbb Z_0$ and $m\in\mathbb N$ such that $\sigma^n(x)=\sigma^{n+m}(x)$.
\end{theorem}
\begin{proof} \emph{Necessity.} Let we have a rational number $x=\frac{u}{v}$, where $u<v$ and $(u,v)=1$. Consider the sequence $(\sigma^n(x))$ generated by the shift operator of expansion \eqref{eq: Cantor series} of the number $x$. That is 
\begin{gather*}
\sigma^0(x)=x,\\
\sigma(x)=q_1x-\varepsilon_1,\\
\sigma^2(x)=q_2\sigma(x)-\varepsilon_2=q_1q_2x-q_2\varepsilon_1-\varepsilon_2,\\
...................................\\
\sigma^n(x)=q_n\sigma^{n-1}(x)-\varepsilon_{n}=x\prod^{n} _{i=1}{q_i}-\left(\sum^{n-1} _{j=1}{\varepsilon_jq_{j+1}q_{j+2}...q_n}\right)-\varepsilon_n\\
...................................
\end{gather*}

From expression \eqref{eq: Cantor series 3} it follows that 
\begin{equation}
\label{eq: Cantor series 4}
\sigma^n(x)=\frac{uq_1q_2...q_n-v(\varepsilon_1q_2...q_n+...+\varepsilon_{n-1}q_n+\varepsilon_n)}{v}=\frac{u_n}{v}.
\end{equation}

Since $v=const$ and the condition $0<\frac{u_n}{v}<1$ holds as $n\to\infty$, we see that $u_n\in\{0,1,...,v-1\}$. Thus there exists a number $m\in\mathbb N$ such that $u_n=u_{n+m}$. In addition, there exists a sequence $(n_k)$ of positive integers such that $u_{n_k}=const$ for all $k\in\mathbb N$. 

\emph{Sufficiency.}  Suppose there exist $n\in \mathbb Z_0$ and $m\in\mathbb N$ such that $\sigma^n(x)=\sigma^{n+m}(x)$. In our case, from expression \eqref{eq: Cantor series 2} it follows that 
$$
x\equiv\Delta^Q _{\varepsilon_1\varepsilon_2...\varepsilon_n000...}+\frac{q_{n+1}...q_{n+m}}{q_{n+1}...q_{n+m}-1}\Delta^Q _{\underbrace{0...0}_n\varepsilon_{n+1}\varepsilon_{n+2}...\varepsilon_{n+m}000...}.
$$

Thus one can to formulate the following proposition.
\begin{lemma}
If there exist numbers $n\in \mathbb Z_0$ and $m\in\mathbb N$ such that $\sigma^n(x)=~\sigma^{n+m}(x)$, then $x$ is rational and the following  expression holds:
$$
\sigma^n(x)=\frac{q_1q_2...q_nq_{n+1}...q_{n+m}}{q_{n+1}...q_{n+m}-1}\Delta^Q _{\underbrace{0...0}_n\varepsilon_{n+1}\varepsilon_{n+2}...\varepsilon_{n+m}000...}.
$$
\end{lemma}
Theorem \ref{th: the main theorem} proved. \end{proof}

\begin{theorem}
\label{th: the main theorem 2}
A number $x=\Delta^Q _{\varepsilon_1\varepsilon_2...\varepsilon_k...}$  is  rational iff there exist numbers $n\in\mathbb Z_0$ and $m\in\mathbb N$ such that
$$
\Delta^{Q} _{\underbrace{0...0}_{n}\varepsilon_{n+1}\varepsilon_{n+2}...}=q_{n+1}...q_{n+m}\Delta^{Q} _{\underbrace{0...0}_{n+m}\varepsilon_{n+m+1}\varepsilon_{n+m+2}...}.
$$
\end{theorem}

Theorem \ref{th: the main theorem} and Theorem \ref{th: the main theorem 2} are equivalent.

\section{Certain corollaries}

Consider the condition $\sigma^n(x)=const$ for all $n\in\mathbb Z_0$. It is easy to see that there exist  numbers $x\in(0;1)$ such that the last-mentioned condition is true, e.g.
$$
x=\frac{1}{3}+\frac{2}{3\cdot 5}+\frac{3}{3\cdot 5\cdot 7}+...+\frac{n}{3\cdot 5\cdot ...\cdot (2n+1)}+...=\sigma^n(x)=\frac{1}{2}.
$$
\begin{lemma}
If $\sigma^n(x)=x$ for all $n\in\mathbb N$, then $\frac{\varepsilon_n}{q_n-1}=const=x$.
\end{lemma}
\begin{proof}
If $\sigma^n(x)=\sigma^{n+1}(x)=const$, then $\sigma^n(x)=q_{n+1}\sigma^n(x)-\varepsilon_{n+1}$. Whence,
\begin{equation}
\label{eq: Cantor series 5}
\sigma^n(x)=\frac{\varepsilon_{n+1}}{q_{n+1}-1}=const.
\end{equation}
That is 
$$
x=\frac{\varepsilon_1}{q_1-1}=\frac{\varepsilon_1}{q_1}+\frac{\varepsilon_2}{q_1(q_2-1)}=...=\sum^{n-1} _{i=1}{\frac{\varepsilon_i}{q_1q_2...q_i}}+\frac{\varepsilon_n}{q_1q_2...q_{n-1}(q_n-1)}=....
$$
\end{proof}
\begin{remark}
If the condition $\sigma^n(x)=const$ holds for all $n\ge n_0$, where $n_0$ is a fixed positive integer number, then from \eqref{eq: Cantor series 3} it follows that the condition $\frac{\varepsilon_n}{q_n-1}=const$ holds for all $n>n_0$.
\end{remark}
\begin{lemma}
Suppose $n_0$ is a fixed positive integer number; then the condition $\sigma^n(x)=const$ holds for all $n\ge n_0$ iff $\frac{\varepsilon_n}{q_n-1}=const$ for all $n>n_0$.
\end{lemma}
\begin{proof}\emph{Necessity} follows from the previous lemma.

\emph{Sufficiency.} Suppose the condition 
$$
const=p=\frac{\varepsilon_n}{q_n-1}=\frac{\varepsilon_{n+1}}{q_{n+1}-1}=...=\frac{\varepsilon_{n+i}}{q_{n+i}-1}=...
$$
 holds for all $n>n_0$. Using the equality $\frac{\varepsilon_n}{q_n}=\frac{\varepsilon_{n}}{q_n-1}-\frac{\varepsilon_{n}}{q_n(q_n-1)}$, we get 
$$
\sigma^n(x)=\sum^{\infty} _{i=n+1}{\frac{\varepsilon_i}{q_{n+1}...q_i}}=\left(\frac{\varepsilon_{n+1}}{q_{n+1}-1}-\frac{\varepsilon_{n+1}}{q_{n+1}(q_{n+1}-1)}\right)+
$$
$$
+\sum^{\infty} _{i=1}{\left(\left(\frac{\varepsilon_{n+i+1}}{q_{n+i+1}-1}-\frac{\varepsilon_{n+i+1}}{q_{n+i+1}(q_{n+i+1}-1)}\right)\prod^{n+i} _{j=n+1}{\frac{1}{q_i}}\right)}=
$$
$$
=p\left(1-\frac{1}{q_{n+1}}\right)+p\sum^{\infty} _{i=1}{\left(\left(1-\frac{1}{q_{n+i+1}-1}\right)\prod^{n+i} _{j=n+1}{\frac{1}{q_j}}\right)}=p.
$$
\end{proof}

It is easy to see that the following statement is true.
\begin{proposition}
\label{pr: Cantor series}
The set $\{x: \sigma^n(x)=x~\forall n\in\mathbb N\}$ is a finite set of order $q=\min_n{q_n}$ and $x=\frac{\varepsilon}{q-1}$, where $\varepsilon\in\{0,1,...,q-1\}$.
\end{proposition}

\begin{lemma}
Suppose we have $q=\min_n{q_n}$ and  fixed $\varepsilon\in\{0,1,...,q-1\}$; the condition $\sigma^n(x)=x=\frac{\varepsilon}{q-1}$ holds iff the condition $\frac{q_n-1}{q-1}\varepsilon=\varepsilon_n\in\mathbb Z_0$ holds for all $n\in\mathbb N$.
\end{lemma}
\begin{proof}\emph{Necessity} follows from Proposition \ref{pr: Cantor series} and equality \eqref{eq: Cantor series 5}.

\emph{Sufficiency.}  Suppose $\varepsilon_n=\frac{q_n-1}{q-1}\varepsilon$; then  
$$
x=\sum^{\infty} _{n=1}{\frac{\frac{q_n-1}{q-1}\varepsilon}{q_1q_2...q_n}}=\frac{\varepsilon}{q-1}\sum^{\infty} _{n=1}{\frac{q_n-1}{q_1q_2...q_n}}=\frac{\varepsilon}{q-1}.
$$
\end{proof}
\begin{corollary}
Let $n_0$ be a fixed positive integer number, $q_0=\min_{n>n_0}{q_n}$ and $\varepsilon_0$ be a numerator of the fraction 
$\frac{\varepsilon_{n_0+k}}{q_1q_2...q_{n_0}q_{n_0+1}...q_{n_0+k}}$ in expansion \eqref{eq: Cantor series} of $x$ providing that $q_{n_0+k}=q_0$; then $\sigma^n(x)=const$ for all $n\ge n_0$ iff the condition $\frac{q_n-1}{q_0-1}\varepsilon_0=\varepsilon_n\in\mathbb Z_0$ holds for any $n>n_0$.
\end{corollary}

Consider expression \eqref{eq: Cantor series 4}. The main attention will give to existence of  the sequence $(n_k)$ such that $u_{n_k}=const$ for all $k\in\mathbb N$ in  \eqref{eq: Cantor series 4}. Let we have a rational number $x\in [0;1]$ then there exists the sequence $(n_k)$ such that $\sigma^{n_k}(x)=const$. The last condition  can be written by 
$$
const=\sum^{\infty} _{i=n_1+1}{\frac{\varepsilon_i}{q_{n_1+1}...q_i}}=\sum^{\infty} _{i=n_2+1}{\frac{\varepsilon_i}{q_{n_2+1}...q_i}}=...=\sum^{\infty} _{i=n_k+1}{\frac{\varepsilon_i}{q_{n_k+1}...q_i}}=....
$$
It is easy to see that  Cantor series \eqref{eq: Cantor series} for which  the condition $\sigma^{n_k}(x)=const$ holds can be written by a certain  Cantor series for which the condition $\sigma^{k}(x)=const$ holds. That is 
$$
x=\sum^{\infty} _{k=1}{\frac{\varepsilon_k}{q_1q_2...q_k}}=\sum^{n_1} _{j=1}{\frac{\varepsilon_j}{q_1q_2...q_j}}+\frac{1}{q_1q_2...q_{n_1}}x^{'},
$$
$$
x^{'}=\sum^{\infty} _{k=1}{\frac{\varepsilon_{n_k+1}q_{n_k+2}q_{n_k+3}...q_{n_{k+1}}+\varepsilon_{n_k+2}q_{n_k+3}...q_{n_{k+1}}+\varepsilon_{n_{k+1}-1}q_{n_{k+1}}+\varepsilon_{n_{k+1}}}{(q_{n_1+1}...q_{n_2})(q_{n_2+1}...q_{n_3})...(q_{n_k+1}...q_{n_{k+1}})}}=
$$
\begin{equation}
\label{eq: Cantor series 6}
=\sum^{\infty} _{k=1}{\frac{\lambda_k}{(\mu_1+1)...(\mu_k+1)}}.
\end{equation}
 In the case of series~\eqref{eq: Cantor series 6} the condition $\sigma^k(x^{'})=const$ holds for all $k=~0,1,...$.

Clearly, the following statement is true.
\begin{theorem}
The number $x$ represented by expansion \eqref{eq: Cantor series} is rational iff there exists a subsequence $(n_k)$ of positive integers sequence such that for all $k=1,2,...,$ the following conditions are true:
\begin{itemize}
\item
$$
\frac{\lambda_k}{\mu_k}=
\frac{\varepsilon_{n_k+1}q_{n_k+2}...q_{n_{k+1}}+\varepsilon_{n_k+2}q_{n_k+3}...q_{n_{k+1}}+...+\varepsilon_{n_{k+1}-1}q_{n_{k+1}}+\varepsilon_{n_{k+1}}}{q_{n_k+1}q_{n_k+2}...q_{n_{k+1}-1}}=const;
$$
\item $\lambda_k=\frac{\mu_k}{\mu}\lambda$, where $\mu=\min_{k\in\mathbb N}{\mu_k}$ and $\lambda$ is a number in the numerator of fraction, which denominator equals $(\mu_1+1)(\mu_2+~1)...(\mu+~1)$,  from sum \eqref{eq: Cantor series 6}.
\end{itemize}
\end{theorem}

Let  $x$ be a rational number then 
$$
x=\sum^{\infty} _{k=1}{\frac{\varepsilon_k}{q_1q_2...q_k}}=
$$
$$
=\sum^{n} _{i=1}{\frac{\varepsilon_i}{q_1q_2...q_i}}+\left(\frac{\varepsilon_{n+1}}{q_1q_2...q_{n+1}}+...+\frac{\varepsilon_{n+m}}{q_1q_2...q_nq_{n+1}...q_{n+m}}\right)+\sum^{\infty} _{j=1}{\frac{\varepsilon_{n+m+j}}{q_1q_2...q_{n+m+j}}}=
$$
$$
=\sum^{n} _{i=1}{\frac{\varepsilon_i}{q_1q_2...q_i}}+\frac{\varepsilon_{n+1}q_{n+2}...q_{n+m}+...+\varepsilon_{n+m-1}q_{n+m}+\varepsilon_{n+m}}{q_1q_2...q_n(q_{n+1}...q_{n+m})}+\sum^{\infty} _{j=1}{\frac{\varepsilon_{n+m+j}}{q_1q_2...q_{n+m+j}}}.
$$
From Theorem  \ref{th: the main theorem} it follows that
$$
\sigma^n(x)=\sigma^{n+m}(x)=\frac{\varepsilon_{n+1}q_{n+2}...q_{n+m}+...+\varepsilon_{n+m-1}q_{n+m}+\varepsilon_{n+m}}{q_{n+1}...q_{n+m}-1}.
$$

Hence the following statement is true.
\begin{theorem}
If the number $x$ represented by Cantor series \eqref{eq: Cantor series} is rational $(x=\frac{u}{v})$, then there exist $n\in\mathbb Z_0$ and $m\in\mathbb N$ such that 
$$
q_1q_2...q_n(q_{n+1}q_{n+2}...q_{n+m}-1)\equiv 0\pmod{v}.
$$
\end{theorem}


\begin{thebibliography}{9}


\bibitem{Cantor1} G.~Cantor, Ueber die einfachen Zahlensysteme,  \emph{ Z. Math. Phys.}, \textbf{14}, 121--128 (1869).


\bibitem{DO1955} P. H. Diananda  and A. Oppenheim, Criteria for irrationality of certain classes of numbers II, \emph{Amer. Math. Monthly}, \textbf{62}, No. 4,  222--225 (1955).
  
\bibitem{ES1968}  P. Erd\"os and  E. G. Straus, On the irrationality of certain Ahmes series, \emph{J. Indian. Math. Soc.}, \textbf{27},  129--133 (1968).

\bibitem{ES1974}  P. Erd\"os and  E. G. Straus, On the irrationality of certain series , \emph{Pacific J. Math.}, \textbf{55}, No.1, 85--92 (1974).

\bibitem{H1997}  J. Han\v{c}l, A note to the rationality of infinite series I, \emph{Acta Math. Inf. Univ. Ostr.}, \textbf{5}, No.1,  5--11 (1997).

\bibitem{H2002}  J. Han\v{c}l, A note on a paper of Oppenheim and \v{S}al\'at concerning series of Cantor type, \emph{Acta Math. Inf. Univ. Ostr.}, \textbf{10},   35--41 (2002).

\bibitem{HT2004}  J. Han\v{c}l and R. Tijdeman,  On the irrationality of Cantor series, \emph{J. Reine Angew. Math.}, \textbf{571},   145--158 (2004).


\bibitem{HT2004(2)}  J. Han\v{c}l and R. Tijdeman,  On the irrationality of Cantor and Ahmes series, \emph{Publ. Math. Debrecen}, \textbf{65},  No.3-4, 371--380 (2004).

\bibitem{HT2005}  J. Han\v{c}l and R. Tijdeman,  On the irrationality of factorial series, \emph{Acta Arith.}, \textbf{118},   383--401 (2005).

\bibitem{HT2009}  J. Han\v{c}l and R. Tijdeman,  On the irrationality of factorial series III, \emph{Indag. Mathem., N. S.}, \textbf{20}, No.4,  537--549 (2009).

\bibitem{HT2010}  J. Han\v{c}l and R. Tijdeman,  On the irrationality of factorial series II, \emph{J. Number Theory}, \textbf{130}, No.3,  595--607 (2010).

\bibitem{KL2001}  P. Kuhapatanakul and V. Laohakosol,   Irrationality of some series with rational terms, \emph{Kasetsart J. (Nat. Sci.)}, \textbf{35},   205--209 (2001).

\bibitem{M2010} B. Mance, \emph{Normal numbers with respect to the Cantor series expansion}, Dissertation, The Ohio State University, 2010.

\bibitem{O1954}  A. Oppenheim,   Criteria for irrationality of certain classes of numbers, \emph{Amer. Math. Monthly}, \textbf{61},  No.4,  235--241 (1954).

\bibitem{S2013}  S. Serbenyuk,   Representation of real numbers by the alternating Cantor series, \emph{arXiv:1602.00743v1}. Link:  http://arxiv.org/pdf/1602.00743v1.pdf

\bibitem{S2012} S. O. Serbenyuk, Real numbers representation by the Cantor series, \emph{ International Conference on Algebra dedicated to 100th anniversary of S. M. Chernikov: Abstracts, Kyiv: Dragomanov National Pedagogical University}, p. 136 (2012). Link: https://www.researchgate.net/publication/301849984 

\bibitem{TY2002}  Robert Tijdeman and Pingzhi Yuan,  On the rationality of Cantor and Ahmes series, \emph{Indag. Math. (N.S.)}, \textbf{13},  No.3, 407--418 (2002).

\end{thebibliography}
\end{document}